\documentclass[reqno]{amsart}
\usepackage{mathrsfs}
\usepackage{}
\usepackage[top=1in, bottom=1in, left=1in, right=1in]{geometry}
\usepackage{cases}
\usepackage[colorlinks,
            linkcolor=red,
            anchorcolor=blue,
            citecolor=green
            ]{hyperref}

\newtheorem{theorem}{Theorem}[section]
\newtheorem{lemma}[theorem]{Lemma}
\theoremstyle{definition}

\newtheorem{prop}[theorem]{Proposition}
\newtheorem{conj}[theorem]{Conjecture}

\newtheorem{cor}[theorem]{Corollary}

\theoremstyle{remark}
\newtheorem{remark}[theorem]{Remark}
\renewcommand{\dim}{\mbox{\rm dim}}

\newcommand{\Gal}{\mbox{\rm Gal}}
\numberwithin{equation}{section}

\begin{document}\large

\title{On the elliptic curve $y^{2}=x^{3}-2rDx$ and Factoring integers }

\author{Xiumei Li, Jinxiang Zeng }
\address{Department of Mathematical Science, Tsinghua University, Beijing 100084, P. R. China}
\email{xm-li09@mails.tsinghua.edu.cn\\ cengjx09@mails.tsinghua.edu.cn}






\begin{abstract}
Let $D=pq$ be the product of two distinct odd primes. Assume the parity conjecture, we construct infinitely many $r\ge 1$ such that $E_{2rD}:y^2=x^3-2rDx$ has conjectural rank one and
$v_p(x([k]Q))\not=v_q(x([k]Q))$ for any odd integer $k$, where $Q$ is the generator of the free part of $E(\mathbb Q)$. Furthermore, under the Generalized Riemann hypothesis, the minimal value of  $r$ is in $\textrm{O}(\log^4D)$. As a corollary, one can factor $D$ by computing the generator $Q$.
\end{abstract}

\maketitle

\section{Introduction and Main Results}

In this paper, we study the elliptic curve $E_{2rD}/\mathbb{Q}$ defined by $$ E_{2rD}:y^2=x^3-2rDx,\ r\in\mathbb{Z}_{\geq1}, $$ where $D=pq$ is a product of two distinct odd primes and $ 2rD $ is square-free.  Burhanuddin and Huang \cite{Burhanuddin} studied $E_D:y^2=x^3-Dx$ with $D=pq,p\equiv q\equiv 3\mod 16$. And they proved that, under the parity conjecture, $E_D$ has conjectural rank one
and $v_p(x(Q))\not=v_q(x(Q))$, where $Q$ is the generator of $E_D(\mathbb Q)/E_D(\mathbb Q)_{\textrm{tors}}$, so one can recover $p$ and $q$ from $D$ and $x(Q)$.
Furthermore, they conjectured that factoring $D$ is polynomial time
equivalent to computing the generator $Q$.

In this paper, we generalize their results to all the $D=pq$ with odd primes $p\not=q$. Moreover, for a given $D$, we prove that there are infinitely many $r\ge1$ such that $E_{2rD}$ has conjectural rank one and $v_p(x([k]Q))\not=v_q(x([k]Q))$ for any odd integer $k$, where $Q$ is the generator of $E_{2rD}(\mathbb{Q})/E_{2rD}(\mathbb{Q})_{\textrm{tors}}$. The main results are summarized as follows,

\begin{theorem}\label{theorem:main1}Let $D=pq$, where $p$ and $q$ are distinct odd primes. Under the parity conjecture, there are infinitely many positive integer $r$ such that $E_{2rD}:y^2=x^3-2rDx$ has conjectural rank one. Let $Q$ be the generator of $E_{2rD}(\mathbb{Q})/E_{2rD}(\mathbb{Q})_{\textrm{tors}}$, then $v_p(x([k]Q))\not=v_q(x([k]Q))$ for any odd integer $k$. Furthermore, assume the Generalized Riemann hypothesis, the minimal value of  $r$ is in $\textrm{O}(\log^4 D)$.
\end{theorem}

\begin{cor}\label{corollary:basic} Given $ D$, a product of two distinct odd primes. Assume the parity conjecture and the Generalized Riemann hypothesis.
Then there exists an odd integer $r\in\textrm{O}(\log^4 D)$ such that $D$ can be recovered from the data $x([2k+1]Q)$ and $D$, where $Q$ is the generator of $E_{2rD}(\mathbb{Q})/E_{2rD}(\mathbb{Q})_{\textrm{tors}}$ and $k$ is any rational integer.
\end{cor}

\begin{remark} Burhanuddin and Huang \cite{Burhanuddin} have proved that, if the naive height of the generator of the free part of $E_D(\mathbb{Q})$ with $D=pq,p\equiv q\equiv 3(\bmod 16)$, grows polynomially in $\log \Delta$, where $\Delta$ is the minimal discriminant of $E_D$, then factoring $D$ is polynomial time reducible to computing the generator of the group $E_D(\mathbb{Q})$. On the other hand, for general $D$, Lang \cite{Lang} conjectured that the log height of the generator is (approximately) bounded above by $D^{1/12}$ and is widely believed that this upper bound is accurate for "most" elliptic curves. More precisely, it is a folklore conjecture that the probability of an elliptic curve $E_D(\mathbb{Q})$ has a generator with height $h(P)\le D^{1/13}$ is less than $D^{-c}$ for some absolute constant $c$. So although the additional parameter $r$ leads to a family of curves which are suitable for factoring $D$, it's unlikely to find one, whose height of generator grows polynomially in $\log \Delta$.
\end{remark}
The paper is organized as follows. In section 2 we use the 2-descent method to compute the Selmer groups of these elliptic curves $E_{2rD}$ and their isogenous curves $E_{2rD}'$,
which will give an upper bound of rank of Mordell-Weil group $E_{2rD}(\mathbb{Q})$, and we construct infinitely many such curves $E_{2rD}$ such that $r_{E_{2rD}}\leq 1$ for an appropriately chosen value of the parameter $r$. In section 3 we calculate the global root number of these elliptic curves $E_{2rD}$, which will give a lower bound of rank of Mordell-Weil group $E_{2rD}(\mathbb{Q})$ under the weak parity conjecture. In section 4 we study the property of non-torsion points of Mordell-Weil group $E_{2rD}(\mathbb{Q})$ such as arithmetic property and canonical height. In section 5 we prove our main results.

\section{Computation of the Selmer groups}

In this section, we determine the Selmer group $S^{(\phi )} (E_{2rD} / \mathbb{Q} )$ of the elliptic curve $E_{2rD}$ defined as in the introduction.

Let
$$E_{2rD}^{\prime }: y^{2} = x^{3} + 8rDx $$
be its 2-isogenous elliptic curve
and $$ \phi : E_{2rD} \longrightarrow E_{2rD}^{\prime }\ \text{and}\
 \widehat{ \phi } : E_{2rD}^{\prime } \longrightarrow E_{2rD}$$ be the corresponding 2-isogeny defined as
\[
\phi((x,y))=\left(\frac{y^2}{x^2},\frac{-y(2rD+x^2)}{x^2}\right) \quad\text{ and }
\quad \widehat{\phi}((x,y))=\left(\frac{y^2}{4x^2},\frac{y(8rD-x^2)}{8x^2}\right)
\]
respectively.

 Let $ S = \{ \infty \} \bigcup \{ \text{primes \ dividing } \ 2rD \}$
and $ \mathbb{Q}(S, 2)=<S>$ be the subgroup of $\mathbb{Q}^*/\mathbb{Q}^{*2}$. For $ d \in \mathbb{Q}(S, 2)$, the corresponding homogeneous spaces
are defined as follows,
\[
C_{d} : \ d W^{2} = d^{2} + 8rDZ^{4}\quad\text{and}\quad
 C_{d}^{\prime} : \ d W^{2} = d^{2} -  2rDZ^{4}.
\]
By Proposition 4.9 in \cite{Silverman1}, the Selmer groups have the following identifications:
$$ \{1,2rD\}\subseteq S^{(\phi )} (E_{2rD} / \mathbb{Q} ) \cong \{ d \in \mathbb{Q}(S, 2) : \ C_{d} (
\mathbb{Q}_{v}) \neq \emptyset  \text{ for all } v \in S \}, $$
$$\{1,-2rD\}\subseteq S^{(\widehat{\phi })} (E_{2rD}^{\prime } / \mathbb{Q} ) \cong \{ d \in \mathbb{Q}(S, 2) : \ C_{d}^{'} (
\mathbb{Q}_{v}) \neq \emptyset  \text{ for all } v \in S \}. $$

For $ d \in \mathbb{Q}(S,2) $ and $2\not|d$, the following two lemmas give the sufficient and necessary conditions for $d\in S^{(\phi )} (E_{2rD} / \mathbb{Q} )$ or $d\in S^{(\widehat{\phi })} (E_{2rD}^{\prime } / \mathbb{Q} )$.

\begin{lemma}\label{prop:mod1} Let $C_{d}:W^{2}=d+\frac{8rD}{d}Z^{4}$ with $ d|rD$, then
\begin{itemize} \item[(1)] $C_{d}(\mathbb{Q}_{2})\neq \emptyset$ if and only if $d \equiv 1(\bmod\ 8)$;
\item[(2)] $C_{d}(\mathbb{Q}_{t})\neq \emptyset$ if and only if $(\frac{d}{t})=1$ for any prime $t|\frac{rD}{d}$;
\item[(3)] $C_{d}(\mathbb{Q}_{l})\neq \emptyset$ if and only if $(\frac{2rD/d}{l})=1$ for any prime $l|d$.
\end{itemize}
\end{lemma}
\begin{proof} For simplicity, let $f(Z,W):=W^{2}-\frac{8rD}{d}Z^{4}-d$.
\begin{itemize}\item[(1)] If $C_{d}(\mathbb{Q}_{2})\neq\emptyset $, take any point $(z,w)\in C_{d}(\mathbb{Q}_{2})$, then $v_{2}(z)\geq 0$, $v_{2}(w)=0$ and $w^{2}=d+\frac{8rD}{d}z^{4}$.
Taking the valuation $v_{2}$ at $2$ of both sides, we have $d\equiv1(\bmod\ 8)$. On the other hand, if $d\equiv1(\bmod\ 8)$, then $v_{2}(f(0,1))>2v_{2}(\frac{\partial{f}}{\partial{W}}(0,1))$. By Hensel's lemma in \cite{Silverman1},
$f(Z,W)=0$ has a solution in $\mathbb{Q}_{2}^{2}$, that is, $C_{d}(\mathbb{Q}_{2})\neq\emptyset $.
\item[(2)] For any prime $t|\frac{rD}{d}$, if $C_{d}(\mathbb{Q}_{t})\neq\emptyset $, take any point $(z,w)\in C_{d}(\mathbb{Q}_{t})$, then $v_{t}(z)\geq 0$, $v_{t}(w)=0$ and $w^{2}=d+\frac{8rD}{d}z^{4}$.
Taking the valuation $v_{t}$ at $t$ of both sides, we have $(\frac{d}{t})=1$. On the other hand, if $(\frac{d}{t})=1$, then there exists an integer $a$ such that $a^{2}\equiv d(\bmod t)$, therefore $v_{t}(f(0,a))>2v_{t}(\frac{\partial{f}}{\partial{W}}(0,a))$. By Hensel's lemma,
$f(Z,W)=0$ has a solution in $\mathbb{Q}_{t}^{2}$, that is, $C_{d}(\mathbb{Q}_{t})\neq\emptyset $.
\item[(3)] For any prime $l|d$, let $g(Z_{1},W_{1};l,i):=W_{1}^{2}-\frac{8rD}{d}Z_{1}^{4}-dl^{4i}$ for some non-negative integer $i$. If $C_{d}(\mathbb{Q}_{l})\neq\emptyset $, take any point $(z,w)\in C_{d}(\mathbb{Q}_{l})$, then $z=l^{-i}z_{1}$, $w=w_{1}^{-2i}$, where $v_{l}(z_{1})=v_{l}(w_{1})=0$, $i\geq0$ and $(z_{1},w_{1})$ satisfying $w_{1}^{2}=dl^{4i}+\frac{8rD}{d}z_{1}^{4}$.
Taking the valuation $v_{l}$ at $l$ of both sides, we have $(\frac{2rD/d}{l})=1$. On the other hand, if $(\frac{2rD/d}{l})=1$, then there exists an integer $b$ such that $b^{2}\equiv \frac{8rD}{d}(\bmod\ l)$, therefore we have $v_{l}(f(0,b))>2v_{l}(\frac{\partial{f}}{\partial{W}}(0,b))$. By Hensel's lemma,
$g(Z_{1},W_{1};l,i)=0$ has a solution in $\mathbb{Q}_{l}^{2}$, that is, $C_{d}(\mathbb{Q}_{l})\neq\emptyset $.
\end{itemize}
\end{proof}

\begin{lemma}\label{prop:mod2} Let $C_{d}':W^{2}=d-\frac{2rD}{d}Z^{4}$ with $ d|rD$, then
\begin{itemize} \item[(1)] $C_{d}'(\mathbb{Q}_{2})\neq \emptyset$ if and only if $ d-2rD/d \equiv 1 (\bmod 8)$ or $d \equiv 1 (\bmod 8)$;
\item[(2)] $C_{d}'(\mathbb{Q}_{t})\neq \emptyset$ if and only if $(\frac{d}{t})=1$ for any prime $t|\frac{rD}{d}$;
\item[(3)] $C_{d}'(\mathbb{Q}_{l})\neq \emptyset$ if and only if $( \frac{-2rD/d}{l})=1$ for any prime $l|d$.
\end{itemize}
\end{lemma}
\begin{proof} The proof is similar to lemma \ref{prop:mod1}.
\end{proof}

We now determine the Selmer groups of the elliptic curves appeared
in our theorems in the introduction, by the lemmas above. For simplicity, the class of odd integer $m$ in $\mathbb Z/8\mathbb Z$ will be denoted by $\overline m$, and define $a_{\overline
3}=a_{\overline 7}=\overline 5$ and $a_{ \overline{5}}=\overline 3$. Define the set $A_D$ of parameters $r$ as follows,

$
A_D=\left\{
        \begin{array}{ll}
         \{l~|~l\in \overline{q},\left(\frac{D}{l}\right)=-(\frac{l}{p})=-1\},&\text{ if }D\equiv7\mod8,p\not\equiv1\mod 8 \\
         \{l~|~Dl\in \overline{1},\left(\frac{D}{l}\right)=-(\frac{l}{p})=-1\},&\hbox{ if }D\equiv3,5\mod8,p\not\equiv1\mod 8 \\
          \{l~|~l\in a_{\overline{q}}, \left(\frac{D}{l}\right)=-1\},&\hbox{ if }D\equiv 1\mod8,p\not\equiv1 \mod 8 \\
          \{l ~|~l\in a_{\overline{q}},\left(\frac{D}{l}\right)=\left(\frac{p}{l}\right)=-1\},&\hbox{ if }D\not\equiv 1 \mod8,p\equiv1 \mod 8 \\
          \{l_1l_2 ~|~l_{1}\in \overline{3},l_{2}\in \overline{7},\left(\frac{D}{l_1}\right)=\left(\frac{D}{l_2}\right)=-1, \left(\frac{l_1l_2}{p}\right)=-1 \},& \hbox{ if }D\equiv 1 \mod8,p\equiv 1\mod 8
        \end{array}
      \right.
$
where $l,l_1,l_2$ are primes. By density theorem, $A_{D}$ is an infinite set. The Selmer group of $E_{2rD}$ is determined by the following proposition.

\begin{prop}\label{them:main3}~~Let $D=pq$ be the product of two distinct odd primes $p,q$. Then
for any $r\in A_D$, we have
   $$  S^{(\phi )}(E_{2rD} /\mathbb{Q} ) \cong \mathbb{Z}/2\mathbb{Z}\quad\text{ and }\quad
   S^{(\widehat{\phi} )}(E_{2rD}^{'}/ \mathbb{Q})  \cong ( \mathbb{Z}/2\mathbb{Z} )^{2}. $$
\end{prop}

\begin{proof} By the lemmas above, we can determine the elements in the Selmer groups of $E_{2rD}$ and $E_{2rD}'$ easily, where
$r=1,l,l_1l_2$. We only take $r=l_1l_2$ for an example.

In this case, $S=\{\infty\}\cup\{2,p,q,l_1,l_2\}$ and $\mathbb Q(S,2)=<-1,2,p,q,l_1,l_2>$, where
$p\equiv q\equiv 1,l_1\equiv 3,l_2\equiv 7\mod 8$ and they satisfying
$(\frac D{l_1})=(\frac D{l_2})=-1$ and $(\frac {l_1}p)(\frac {l_2}p)=-1$. Without loss of generality, we can
assume $(\frac {l_1}p)=1,(\frac {l_2}p)=-1$ and $(\frac {l_1}{l_2})=1$. By Lemma \ref{prop:mod1},
$S^{(\phi )}(E_{2rD} /\mathbb{Q} )\cong\{1, 2Dl_1l_2\}\cong \mathbb{Z}/2\mathbb{Z}$.

For $S^{(\widehat{\phi} )}(E_{2rD}^{'}/ \mathbb{Q})$, by Lemma \ref{prop:mod2}, one can check
\[
S^{(\widehat{\phi} )}(E_{2rD}^{'}/ \mathbb{Q})=\left\{
        \begin{array}{ll}
         \{ 1, -pl_{2}, 2ql_{1}, -2rD \}, & \hbox{ if }(\frac{q}{p} ) = 1 \\
         \{ 1, -2pl_{1}, ql_{2}, -2rD \} , & \hbox{ if }(\frac{q}{p} ) = -1.
        \end{array}
      \right.
\]
It follows that $S^{(\widehat{\phi} )}(E_{2rD}^{'}/ \mathbb{Q})  \cong ( \mathbb{Z}/2\mathbb{Z} )^{2}$.
\end{proof}

Use the exact sequences listed in \cite{Silverman1}, we have

$$  \begin{array}{l} 0
\longrightarrow \frac{E_{2rD}^{'}(\mathbb{Q} )}{\phi (E_{2rD}(\mathbb{Q}))} \longrightarrow
S^{(\phi )}(E_{2rD} /\mathbb{Q} ) \longrightarrow \text{TS}(E_{2rD} /\mathbb{Q})[\phi ]
\longrightarrow 0 \\
0 \longrightarrow \frac{ E_{2rD} (\mathbb{Q} )}{\widehat{\phi }(E_{2rD}^{'}(\mathbb{Q} ))}
\longrightarrow S^{( \widehat{\phi })}(E_{2rD}^{'}/\mathbb{Q}  )
\longrightarrow \text{TS}( E_{2rD}^{'}/\mathbb{Q}  )[\widehat{\phi } ]
\longrightarrow 0 \\
0 \longrightarrow \frac{E_{2rD}^{'}(\mathbb{Q} )[\widehat{ \phi }]}{\phi
(E_{2rD}(\mathbb{Q})[2])} \longrightarrow \frac{E_{2rD}^{'}( \mathbb{Q} )}{\phi ( E_{2rD} ( \mathbb{Q} )) }
\longrightarrow \frac{E_{2rD}( \mathbb{Q} )}{2E_{2rD}( \mathbb{Q} )} \longrightarrow
\frac{E_{2rD}( \mathbb{Q} )}{\widehat{ \phi }(E_{2rD}^{'}( \mathbb{Q} ))} \longrightarrow  0 \quad
\end{array} $$
which follows that
\[
r_{E_{2rD}} + \dim_{2}(\text{TS}( E_{2rD} / \mathbb{Q}
)[\phi ] ) + \dim_{2}(\text{TS}(E_{2rD}^{'}/\mathbb{Q} )[\widehat{\phi }] ) \]
\[ =\dim_{2}(S^{(\phi )}(E_{2rD} /\mathbb{Q} ) ) + \dim_{2}(S^{(\widehat{\phi
} )}(E_{2rD}^{'}/ \mathbb{Q}) ) - 2,\]
where dim$_2$ denotes the dimension of an $\mathbb{F}_2$-vector space. The equality above gives an upper bound of $r_{E_{2rD}}$ as follows,
\[
r_{E_{2rD}}\leq \dim_{2}(S^{(\phi )}(E_{2rD} /\mathbb{Q} ) ) + \dim_{2}(S^{(\widehat{\phi
} )}(E_{2rD}^{'}/ \mathbb{Q}) ) - 2\]
which implies that \begin{equation}\label{inequality}r_{E_{2rD}}\leq 1,\ \text{for any}\ r\in A_{D}. \end{equation}
In the next section, we study the global root number of $E_{2rD}$, which will give a lower bound of $r_{E_{2rD}}$.

\section{Computation of the conjectural rank}

Let $ E $ be an elliptic curve over $ \mathbb{Q} $ with conductor $ N_{E}. $ By the Modularity Theorem \cite{Edixhoven},
the L-function admits an analytic continuation to an entire function satisfying the functional equation
$$ \Lambda_{E}( 2- s ) = W( E )\Lambda_{E}( s ),  \ \ \text{ where } \Lambda_{E}( s ) = N_{E}^{s/2}( 2\pi )^{-s}\Gamma( s )L_{E}( s ), $$
and $ W( E ) = \pm 1 $  is called the global root number.\\
\indent Let $ r_{E}^{an} $ and $ r_{E} $ be the analytic rank and arithmetic rank of $ E $ respectively, where $ r_{E}^{an} $ is the order of
vanishing of $ L_{E}( s ) $ at $ s = 1 $ and $ r_{E} $ is the rank of the abelian group $ E( \mathbb{Q} ). $ For the parity of $ r_{E} $, there is a famous conjecture:
\begin{conj}\label{conj:main}( \textbf{Parity Conjecture} )  $ (-1)^{r_{E}} = W( E ) $. \end{conj}

\begin{remark} In this paper, we only need the weak form of the parity conjecture: \[W( E ) = -1 \Longrightarrow r_{E} \geq 1.\]
 For the detail of the parity conjecture, see the recent works by Tim Dokchiter \cite{Dokchitser}. \end{remark}
This weak parity conjecture gives a lower bound of $r_{E}$, while $W( E ) = -1$.
Due to the special choice of our elliptic curve, the global root number is equal
to $-1$ by the following lemma.
\begin{lemma}\label{lemma:basic} Let $ E_{2N}: y^{2} = x^{3} - 2Nx $ be an elliptic curve over $ \mathbb{Q} $ with $ 2N $
square-free, then $ W(E_{2N}) = -1. $ \end{lemma}
\begin{proof} By \cite{Birch}, for any integer $d$ such that $d\not\equiv0\mod 4$, the global root number of the elliptic curve
$E_d:y^2=x^3-dx$, has the following formula,
$$W(E_{d})=\text{sgn}(-d)\cdot \epsilon (d)\cdot \prod_{p^2||d,p\ge 3}\left( \frac{-1}{p}\right)$$
where
\begin{displaymath}
\epsilon(d)=\left\{\begin{array}{ll}
-1 &  \textrm{ if $d\equiv $ 1, 3, 11, 13(mod 16)}\\
1  &  \textrm{ if $d\equiv$ 2, 5, 6, 7, 9, 10, 14, 15( mod 16),}
\end{array}
\right.
\end{displaymath}
which tells us $W(E_{2N}) = -1. $ \end{proof}

\begin{cor}\label{corollary:rank} Assume the Parity conjecture, then for any $r\in A_D$ we have,
$$ r_{E_{2rD}} = 1\quad\text{ and }\quad  \text{TS}((E_{2rD}^{'}/\mathbb{Q} )[\widehat{\phi }] ) = 0. $$
\end{cor}
\begin{proof} Inequality \ref{inequality}, lemma \ref{lemma:basic} and the parity conjecture imply the conclusion.
 \end{proof}
Section 2, section 3 and the density theorem state that there are infinitely many integers $r$ such that $E_{2rD}$ have conjectural rank one. For these constructed elliptic curves $E_{2rD}$, Mordell-Weil theorem tells us that \[E_{2rD}(\mathbb{Q})\cong E_{2rD,\text{tors}}(\mathbb{Q})\oplus \mathbb{Z}.\]

And Proposition 6.1 in \cite{Silverman1} tells us that $E_{2rD,\text{tors}}(\mathbb{Q})$ is generated by the 2-order point $T=(0,0)$, that is, $E_{2rD,\text{tors}}(\mathbb{Q})=\{O,T\}$. In the next section, we study the arithmetic property and the canonical height of non-torsion points of $E_{2rD}(\mathbb{Q})$.

\section{the Property of non-torsion points}

For any finite place $v\in M_{\mathbb{Q}}$, let $\mathbb{Q}_{v}$ be the local field, $\mathscr{M}$ be the maximal ideal and $\mathbb{F}_{v}$ be the residue field of $\mathbb{Q}$ at $v$. We define
\[ \widetilde{E}_{\text{ns}}(\mathbb{F}_{v}):=\{P\in \mathbb{F}_{v}):\tilde{P}\ \text{is non-singular in}\ \widetilde{E}(\mathbb{F}_{v}) \}\]
\[ E_{0}(\mathbb{Q}_{v}):=\{P\in \mathbb{Q}_{v}:\tilde{P}\in \widetilde{E}_{\text{ns}}(\mathbb{F}_{v}) \}\]
\[ E_{1}(\mathbb{Q}_{v}):=\{P\in \mathbb{Q}_{v}:\tilde{P}= \tilde{O}\}\]
\[ \widehat{E}(\mathscr{M}):=\text{the formal group associated to}\ E.\]
\begin{lemma}\label{lemma:order} For any $r\in A_D$, let $T=(0,0)$ be the generator of $E_{2rD}(\mathbb{Q})_{\text{tors}}$ and $Q$ the generator of $ E_{2rD}(\mathbb{Q})/E_{2rD}(\mathbb{Q})_{\text{tors}}$, then for any odd prime $t|rD$, the following holds,
\begin{itemize}\item[(1)] if $ v_{t}( x(Q) ) \geq 1$, then for any nonzero integer $k$,
\[ \left \{
   \begin{array}{l}
   0 \geq v_{t}( x( [ 2k ]Q ) ) \equiv 0 (\bmod\ 2)  \\
  1 \leq v_{t}( x( [ 2k + 1 ]Q ) )  \equiv 1 (\bmod\ 2) \\
  1 \leq v_{t}( x( [ 2k ]Q + T ) ) \equiv 1 (\bmod\ 2)\\
  -2 \geq  v_{t}( x( [ 2k + 1 ]Q + T ) ) \equiv 0 (\bmod\ 2).
    \end{array}
  \right.\]
\item[(2)] if $ v_{t}( x(Q) ) \leq 0$, then for any nonzero integer $k$,
\[  \left \{
   \begin{array}{l}
   0 \geq v_{t}( x( [ 2k ]Q ) ) \equiv 0 (\bmod\ 2)   \\
  0 \geq v_{t}( x( [ 2k + 1 ]Q ) )  \equiv 0 (\bmod\ 2) \\
 1 \leq v_{t}( x( [ 2k ]Q + T ) ) \equiv 1 (\bmod\ 2) \\
  1 \leq  v_{t}( x( [ 2k + 1 ]Q + T ) ) \equiv 1 (\bmod\ 2).
    \end{array}
  \right.    \] \end{itemize}
\end{lemma}
\begin{proof}\begin{itemize} \item[ Case I.] For $ v_{t}( x(Q) ) = a \geq 3 $ (note that $ a $ is odd) :  Using Tate's algorithm in \cite{Silverman2},
we know the Tamagawa number $ c_{t}$ is equal to 2, which means
$ E_{2rD}( \mathbb{Q}_{t} )/E_{2rD, 0}( \mathbb{Q}_{t} ) \cong \mathbb{Z}/2\mathbb{Z}. $ Since $ Q \notin E_{2rD, 0}( \mathbb{Q}_{t} ),$ we have
$$ [ 2k ]Q \in E_{2rD, 0}( \mathbb{Q}_{t} ), [ 2k + 1 ]Q  \in E_{2rD}( \mathbb{Q}_{t} ) \setminus E_{2rD, 0}( \mathbb{Q}_{t} ). $$
By the duplication formula,
$$ x( [2]Q ) = \frac{x(Q)^{4} + 4rDx(Q)^{2} + 4(rD)^{2}}{4x(Q)^{3} -8rDx(Q)} $$
so $ v_{t}( x( [ 2 ]Q ) )= - ( a - 1 ) < 0, $ hence $ [ 2 ]Q \in E_{2rD, 1}( \mathbb{Q}_{t} )$.
 The isomorphism $ E_{2rD, 1}( \mathbb{Q}_{t} ) \cong \widehat{ E}_{2rD}(\mathscr{M}), ( x, y )
 \mapsto (z(x)=-\frac{x}{y}, w(z))$ together with  Proposition 2.3 in \cite{Silverman1}
implies that $ v_{t}( z([2k]Q) )=v_{t}( 2k ) + \frac{a-1}{2}$. So
\[  v_{t}( x( [ 2k ]Q ) ) = -2v_{t}( 2k ) - ( a - 1 ), \]
hence
$ v_{t}( x( [ 2 ( 2k + 1 ) ]Q ) )= -2 v_{t}( 2k + 1 ) - ( a - 1)$,
again the duplication formula gives the equality
$ v_{t}( x( [ 2 ( 2k + 1 ) ]Q ) )=1- v_{t}( x( [ 2k + 1 ]Q ) )$,
thus $$ v_{t}( x( [ 2k + 1 ]Q ) ) = 2 v_{t}( 2k + 1 ) + a. $$

On the other hand, by the chord and tangent principal, we have
$$ x( R )\cdot x( R + T ) = -2rD, R \in E_{2rD}( \mathbb{Q}  ) \setminus \{ O, T \} $$
so $$ v_{t}( x( [ 2k ]Q + T ) ) = 2v_{t}( 2k ) +  a, $$
$$ v_{t}( x( [ 2k + 1 ]Q + T ) ) = -2v_{t}( 2k + 1 ) -  ( a - 1 ). $$
By the argument above, we have
$$   \left \{
   \begin{array}{l}
   -2 \geq v_{t}( x( [ 2k ]Q ) ) \equiv 0 (\bmod\ 2)  \\
  3 \leq v_{t}( x( [ 2k + 1 ]Q ) )  \equiv 1 (\bmod\ 2) \\
  3 \leq v_{t}( x( [ 2k ]Q + T ) ) \equiv 1 (\bmod\ 2) \\
  -2 \geq  v_{t}( x( [ 2k + 1 ]Q + T ) ) \equiv 0 (\bmod\ 2).
    \end{array}
  \right.    $$

\item[ Case II.] For $v_{t}( x(Q) ) =1: $ the duplication formula gives $ v_{t}( x( [ 2 ]Q ) ) = 0, $ so
$ \widetilde{E}_{ns}(\mathbb{F}_{t}) = < \widetilde{[2]Q} >, $ where $ \widetilde{[ 2 ]Q} $ is the reduction of $ [2 ]Q $
at the place $ t$. Since $ | \widetilde{E}_{ns}(\mathbb{F}_{t}) | = t$, hence $[t]( \widetilde{[2]Q} ) = O$.
Let $ b:= -v_{t}( x( [ 2t ]Q )$, then $ b$ is an even integer bigger than one. Similarly, we have
$ v_{t}( x( [ t ]Q ) = 1 + b. $ Since $ \widetilde{E}_{ns}(\mathbb{F}_{t}) = < \widetilde{[2]Q} >$, for any integer $k$ not divided by $t$, the valuation $ v_{t}( x( [ k ]( [ 2 ]Q ) )$ is equal to zero.
While $ t|k, $ similarly to the Case I, we have
 $$ v_{t}( x( [2k ]Q ) = -2 v_{t}( 2k ) - ( b - 2 ). $$
Thus for any non-zero integer $ k, $ we have
$$ v_{t}( x( [2k ]Q ) =  \left \{
   \begin{array}{ll}
  0  & \hbox { if }\quad t \not| k \\
  -2 v_{t}( 2k ) - ( b - 2 )& \hbox { if }\quad t | k,
  \end{array}
  \right.$$
$$
 v_{t}( x( [2k + 1 ]Q ) =  \left \{
   \begin{array}{ll}
  1 & \hbox { if }\quad t \not| 2k + 1 \\
  2 v_{t}( 2k + 1 ) + ( b - 1 )  &\hbox { if }\quad t | 2k + 1,
  \end{array}
  \right.    $$
$$ v_{t}( x( [2k ]Q + T ) =  \left \{
   \begin{array}{ll}
  1 &\hbox{ if }\quad  t \not| k \\
  2 v_{t}( 2k ) + ( b - 1 ) &\hbox{ if }\quad t | k,
  \end{array}
  \right.$$
and
$$ v_{t}( x( [2k + 1 ]Q + T ) =  \left \{
   \begin{array}{ll}
  0 &\hbox{ if }\quad  t \not| 2k + 1\\
  -2v_{t}( 2k + 1) - ( b - 2 ) &\hbox{ if }\quad t | 2k + 1.
  \end{array}
  \right.    $$
so
 $$   \left \{
   \begin{array}{l}
   0 \geq v_{t}( x( [ 2k ]Q ) ) \equiv 0 (\bmod\ 2)   \\
  1 \leq v_{t}( x( [ 2k + 1 ]Q ) )  \equiv 1 (\bmod\ 2) \\
  1 \leq v_{t}( x( [ 2k ]Q + T ) ) \equiv 1 (\bmod\ 2) \\
  -2 \geq  v_{t}( x( [ 2k + 1 ]Q + T ) ) \equiv 0 (\bmod\ 2).
    \end{array}
  \right.    $$
From Case I and Case II, we know that if $ v_{t}( x(Q) ) \geq 1, $ then
\begin{equation}\label{equation:basic1}  \left \{
   \begin{array}{l}
   0 \geq v_{t}( x( [ 2k ]Q ) ) \equiv 0 (\bmod\ 2)  \\
  1 \leq v_{t}( x( [ 2k + 1 ]Q ) )  \equiv 1 (\bmod\ 2) \\
  1 \leq v_{t}( x( [ 2k ]Q + T ) ) \equiv 1 (\bmod\ 2)\\
  -2 \geq  v_{t}( x( [ 2k + 1 ]Q + T ) ) \equiv 0 (\bmod\ 2).
    \end{array}
  \right.   \end{equation}

\item[ Case III.] For $ v_{t}( x(Q) ) = 0: $ The nonsingular part of the reduction curve is generated by one point i.e. $ \widetilde{E}_{2rD,ns}(\mathbb{F}_{t}) = < \widetilde{Q} >, $ where $ \widetilde{Q} $
is the reduction of $ Q $. Similarly to the discussion in the Case II,
for any non-zero integer $ k$, we have
 $$   \left \{
   \begin{array}{l}
   0 \geq v_{t}( x( [ 2k ]Q ) ) \equiv 0 (\bmod\ 2)   \\
  0 \geq v_{t}( x( [ 2k + 1 ]Q ) )  \equiv 0 (\bmod\ 2) \\
  1 \leq v_{t}( x( [ 2k ]Q + T ) ) \equiv 1 (\bmod\ 2) \\
  1 \leq  v_{t}( x( [ 2k + 1 ]Q + T ) ) \equiv 1 (\bmod\ 2).
    \end{array}
  \right.    $$
\item[ Case IV.] For $ v_{t}( x(Q) ) \leq -2: $  Similarly to Case I. we have
$$   \left \{
   \begin{array}{l}
   -2 \geq v_{t}( x( [ 2k ]Q ) ) \equiv 0 (\bmod\ 2)   \\
  -2 \geq v_{t}( x( [ 2k + 1 ]Q ) )  \equiv 0 (\bmod\ 2) \\
 3 \leq v_{t}( x( [ 2k ]Q + T ) ) \equiv 1 (\bmod\ 2) \\
  3 \leq  v_{t}( x( [ 2k + 1 ]Q + T ) ) \equiv 1 (\bmod\ 2).
    \end{array}
  \right.    $$
From Case III and Case IV, we know that if $ v_{t}( x(Q) ) \leq 0, $ then
\begin{equation}\label{equation:basic2}  \left \{
   \begin{array}{l}
   0 \geq v_{t}( x( [ 2k ]Q ) ) \equiv 0 (\bmod\ 2)   \\
  0 \geq v_{t}( x( [ 2k + 1 ]Q ) )  \equiv 0 (\bmod\ 2) \\
 1 \leq v_{t}( x( [ 2k ]Q + T ) ) \equiv 1 (\bmod\ 2) \\
  1 \leq  v_{t}( x( [ 2k + 1 ]Q + T ) ) \equiv 1 (\bmod\ 2).
    \end{array}
  \right.    \end{equation}
\end{itemize}
\end{proof}
Using lemma \ref{lemma:order}, one can obtain the following observation.
\begin{cor}\label{corollary:order} Assume the Parity conjecture, then for any $r\in A_D$,
\[ v_p(x([k]Q))\not=v_q(x([k]Q)),\ \ \text{for any odd integer}\ k,\]
where $Q$ is the generator of $E_{2rD}(\mathbb{Q})/E_{2rD}(\mathbb{Q})_{\textrm{tors}}$.
\end{cor}
\begin{proof}We first claim that, there exists two points $R_1,R_2\in E_{2rD}(\mathbb Q)$
such that
\begin{equation}\label{equation:basic3}
v_{p}( x( R_{i} )) + 1 \equiv v_{q}( x( R_{i} )) (\bmod\ 2) \quad\text{ and }\quad
 v_{p}( x( R_{i} ))\cdot v_{q}( x( R_{i})) \leq 0,\ \text{for}\ i=1,2.
 \end{equation}
We take the case $( p, q ) \equiv ( 5, 5 ) \mod 8 $ and
$( \frac{q}{p} ) = 1$  for example. The proof for other cases are similar.

In this case, the prime $ l$ is in the class $a_{\overline q}=\overline 3$ and $( \frac{D}{l} ) = -1$.
For the elliptic curve $ E_{2lD} $, we have $\text{TS}(E_{2lD}^{'}/\mathbb{Q} )[\widehat{\phi }] ) = 0 $ and
$ S^{(\widehat{\phi} )}(E_{2lD}^{'}/ \mathbb{Q}) = \{ 1, -q, 2pl, -2Dl \}$. Thus
$ C_{-q}^{'}( \mathbb{Q} ) \neq \emptyset, C_{2pl}^{'}( \mathbb{Q} ) \neq \emptyset. $  Take
$ ( z_{1}, w_{1} ) \in C_{2pl}^{'}( \mathbb{Q} )$ and $( z_{2}, w_{2} ) \in C_{-q}^{'}( \mathbb{Q} )$.
One can see $ v_{p}( z_{1} ) \leq 0, v_{q}( z_{1} ) \geq 0 $ and $ v_{p}( z_{2} ) \geq 0, v_{q}( z_{2} ) \leq 0. $
Denote the images of $(z_1,w_1)$ and $(z_2,w_2)$ under the following isomorphism
    $$ C_{d}^{'} \longrightarrow E_{2rD}, \quad (z, w ) \longmapsto
( \frac{d}{z^{2}}, \frac{dw}{z^{3}} ), $$
by $R_1$ and $R_2$ respectively. So $ R_{1} := ( \frac{2pl}{z_{1}^{2}}, \frac{2plw_{1}}{z_{1}^{3}} ) $
and $ R_{2}: = ( \frac{-q}{z_{2}^{2}}, \frac{-qw_{2}}{z_{2}^{3}} )$.
Thus $R_1,R_2$ satisfy the relations (\ref{equation:basic3}).

Now, by Lemma \ref{lemma:order}, we have $ v_{p}( x(Q) ) \geq 1, v_{q}( x(Q)
)\leq 0 $ or $ v_{p}( x(Q) ) \leq 0, v_{q}( x(Q) ) \geq 1$. Hence, for any $ k \in \mathbb{Z}$,  (\ref{equation:basic1}) and (\ref{equation:basic2}) give
 $$ v_{p}( x([ 2k + 1 ]Q) ) \neq v_{q}( x([ 2k + 1 ]Q) ).$$

\end{proof}

\begin{remark} Corollary \ref{corollary:order} implies that, given the generator $Q$ of $E_{2rD}(\mathbb{Q})/E_{2rD}(\mathbb{Q})_{\textrm{tors}}$, one can recover the prime divisor $p$ from $x([k]Q)$ and $D$. The complexity of searching rational points of elliptic curves over $\mathbb{Q}$ depends on the height of the rational points. For the elliptic curve $E_{2rD}$, the canonical height of rational points is known as follows,\end{remark}
\begin{prop}\label{prop:height} For any $r\in A_D$, let  $ P \in E_{2rD}(\mathbb{Q}) $ be a non-torsion rational point, and write the $x-$ coordinate of $P$ as $x = \frac{a}{d^2}$. Then,
\begin{equation}\label{equation:height1} \widehat{h}(P) \geq \frac{1}{16}\log(4rD), \end{equation}
\begin{equation}\label{equation:height2} \frac{1}{4}\log(\frac{a^2+2rDd^4}{2rD})\leq \widehat{h}(P) \leq \frac{1}{4}\log(a^2+2rDd^4)+\frac{1}{12}\log2. \end{equation}
\end{prop}
\begin{proof} We first estimate the archimedean contribution $\widehat{h}_\infty$ to the canonical height by using Tate's series. Arguing as Proposition 2.1 in \cite{Bremner} we have
\begin{equation}\label{equation:basic4} 0 \leq \widehat{h}_\infty(P)-\frac{1}{4}\log(x(P)^2+2rD)+\frac{1}{12}\log|\Delta| \leq \frac{1}{12}\log2.
\end{equation}
Let $P_1=2P$, from the Lemma \ref{lemma:order}, it follows that $P_1 \in E_{2rD,0}(\mathbb{Q}_t)$ for any rational prime $t$, the local height of $P_1$ at $t$ is given by the formula Theorem 4.1 in \cite{Silverman2}

\[ \widehat{h}_t(P)=\frac{1}{2}\max\{0,-\text{ord}_{t}(x(P_1))\log t\} \]
\makeatletter
\makeatother
\begin{subnumcases}{+\frac{1}{12}\text{ord}_{t}(\Delta)\log t-}
0, &$ P \in E_{2rD,0}(\mathbb{Q}_t)$, \\
\frac{1}{4}\text{ord}_{t}(2rD)\log t, &$P \notin E_{2rD,0}(\mathbb{Q}_t)$.
\end{subnumcases}
\makeatletter\let\@alph\@@@alph\makeatother

\[ \widehat{h}_t(P_1)=\frac{1}{2}\max\{0,-\text{ord}_{t}(x(P_1))\log t\}+\frac{1}{12}\text{ord}_{t}(\Delta)\log t. \] Writing $x(P_1)=\frac{\alpha}{\delta^2}$, and combining (\ref{equation:basic4}) we obtain:
\[\widehat{h}(P_1)\geq \frac{1}{4}\log(\alpha^2+2rD\delta^4)\]
Since $P_1 \in E_{2rD,0}(\mathbb{R})$, then $\alpha/\delta^2 \geq \sqrt{2rD}$, so in fact $\alpha^2+2rD\delta^4 \geq 4rD\delta^4 \geq 4rD$. This gives the lower bound \[\widehat{h}(P_1)\geq \frac{1}{4}\log(4rD)\] and then the formula $\widehat{h}(P_1)=\widehat{h}(2P)=4\widehat{h}(P)$ gives the desired inequality as Proposition \ref{prop:height}.

In order to prove the upper and lower bounds, rewrite the inequalities (\ref{equation:basic4}), (4.7a) and (4.7b) respectively, as
\[ 0 \leq \widehat{h}_\infty(P)-\frac{1}{4}\log(x(P)^2+2rD)+\frac{1}{12}\log|\Delta| \leq \frac{1}{12}\log2  \]
\[ -\frac{1}{4}\text{ord}_{t}(2rD)\log t \leq \widehat{h}_t(P)-\text{ord}_{t}(d)\log t-\frac{1}{12}\text{ord}_{t}(\Delta)\log t \leq 0  (t\geq2)\]
Adding these estimates over all places yields
\[ -\frac{1}{4}\log (2rD) \leq \widehat{h}(P)-\frac{1}{4}\log(a^2+2rDd^4)\leq \frac{1}{12}\log 2.\]
Therefore the required inequality follows from the inequality above.

\end{proof}
\begin{remark} The lower bound (\ref{equation:height1}) is a special case of a conjecture of Serge Lang \cite{Lang}, which says that the canonical height of a non-torsion point on elliptic curve should satisfy \[\widehat{h}(P)>> \log|\Delta|,\] where $\Delta=2^{9}(rD)^{3}$ is the discriminant of $E_{2rD}$. \end{remark}
\begin{remark} Using inequality (\ref{equation:height2}), one can obtain the difference between the canonical height and the Weil height such as \[ -\frac{1}{4}\log(4rD)\leq \widehat{h}(P)-\frac{1}{2}h(P) \leq \frac{1}{4}\log(2rD+1)+\frac{1}{12}\log 2, \]
where $h(P)$ is the Weil height of $P$ defined as $h(P)=\frac{1}{2}\log\max\{|a|,|d^{2}|\}$. Thus if one can give an efficient and practical algorithm for an upper bound estimate of the canonical height, then by Zagier's Theorem 1.1 in \cite{Siksek} one can find the generator of $E_{2rD}$.  \end{remark}

\section{the proof of the main results}

Now let's come to prove our main results theorem \ref{theorem:main1} and Corollary \ref{corollary:basic}.

\begin{proof}[{\bf Proof of theorem \ref{theorem:main1}}] The definition of $A_{D}$, Corollary \ref{corollary:rank}, Corollary \ref{corollary:order} and the density theorem state that there are infinitely many integers $r$ such that $E_{2rD}$ has conjectural rank one and \[ v_p(x([k]Q))\not=v_q(x([k]Q)),\ \ \text{for any odd integer}\ k,\]
where $Q$ is the generator of $E_{2rD}(\mathbb{Q})/E_{2rD}(\mathbb{Q})_{\textrm{tors}}$.

The following lemma shows that the least additional parameter $r$ in $A_D$
has a small upper bound.
\begin{lemma}\label{lemma:minimal}
Given positive integers $k,m$, there exits a constant $c$ dependent only on $m$ and $k$ such that
for any sequences of pairwise different odd primes
$p_1,\ldots,p_k,$ signs $\epsilon_1,\ldots,\epsilon_k\in\{-1,1\},$ and an integer $a$ coprime to $m$, where $p_i\not| m$
for all $i$,
the least odd prime $l$
satisfying
$$\left( \frac{p_i}{l}\right)= \epsilon_i,\  1\le i\le k,\quad and\quad l\equiv a \mod m $$
is upper bounded by $c\log^2(\prod_{i=1}^k p_{i}).$

\end{lemma}
proof.
Let
\begin{displaymath}
d_i = \left\{ \begin{array}{ll}
p_i & \textrm{if $p_{i}\equiv 1\mod 4$}\\
4p_i & \textrm{otherwise}
\end{array} \right.,\quad (1\le i\le k)
\end{displaymath}
and $K_{i}=\mathbb{Q}(\sqrt{d_i}),1\le i\le k$. Then $\text{Gal}(K_i/\mathbb{Q})\cong\{1,-1\}$.
Identifying the two groups, we have  $\left(\frac{p_i}{l}\right)=\left(\frac{K_i/\mathbb{Q}}{l}\right)$.
The latter is the Artin symbol.

Let $K_0=\mathbb{Q}(\zeta_m)$, where $\zeta_m$ is a $m$-th primary root of unity. We have
$\text{Gal}(K_0/\mathbb{Q})\cong (\mathbb{Z}/m\mathbb{Z})^*$,
 and $l\equiv a\mod m$ if and only if $\left(\frac{K_0/\mathbb{Q}}{l}\right)=a\mod m$.

Let $K=\mathbb{Q}(\sqrt{p_1},\ldots,\sqrt{p_k},\zeta_m)$,  which is an abelian extension of $\mathbb{Q}$.
An odd prime $l$ is un-ramified in $K$ if and only if $l\not| m\prod_{i=1}^kd_i$. For such $l$, we have
$$\left(  \frac{K/\mathbb{Q}}{l} \middle)\right|_{K_i}= \left(  \frac{K_i/\mathbb{Q}}{l} \right)\quad (0\le i\le k).$$
Since $p_1,\cdots,p_k$ are pairwise different and $(p_1\cdots p_k,m)=1$, we have $K_i\cap K_j=\mathbb Q$ for
$0\le i\not=j\le k$, and so we have the isomorphism
\[
\Gal(K/\mathbb Q)\cong\prod_{0\le i\le k}\Gal(K_i/\mathbb Q)\quad \sigma\mapsto (\sigma |_{K_i})_{0\le i\le k}.
\]
Thus there exists a unique $\sigma\in\text{Gal}(K/\mathbb{Q})$ such that
$\sigma|_{K_i}=\epsilon_i$ for $1\le i\le k$ and $\sigma|_{K_0}=a\mod m$.
By Theorem 3.1(3) in \cite{Bach}, there exists a prime $l \le (1+o(1))\log^2 (|\Delta_K|)$ such that
$\left(  \frac{K/\mathbb{Q}}{l} \right)=\sigma$, where $\Delta_K$ is the discriminant of $K$. For fixed $k$ and $m$,
one can get easily that
there exists a constant $c'>0$ satisfying $|\Delta_K|<c'(\prod_{i=1}^k p_i)^{m2^{k-1}}$. Thus we can find a constant $c$,
dependent on $m$ and $k$, such that there exists a prime $l<c\log^2(\prod_{i=1}^kp_i)$ satisfying
$\left(\frac{p_i}{l}\right)=\epsilon_i$ for $1\le i\le k$, and $l\equiv a\mod m$.

Lemma \ref{lemma:minimal} shows that there exists an integer $r\in O(\log^4D)$ such that the elliptic curve $E_{2rD}$ has conjectural rank one and \[ v_p(x([k]Q))\not=v_q(x([k]Q)),\ \ \text{for any odd integer}\ k,\]
where $Q$ is the generator of $E_{2rD}(\mathbb{Q})/E_{2rD}(\mathbb{Q})_{\textrm{tors}}$.

\end{proof}

\begin{proof}[{\bf Proof of Corollary \ref{corollary:basic}}]
It follows directly from the theorem \ref{theorem:main1}.
\end{proof}

\begin{remark}
Theorem \ref{theorem:main1} tells us that any point of infinite order in $E_{2rD}(\mathbb{Q})$ can be used to factor $D$ (if the point is of even order, use the duplication formula to "halving" it). Since rank$(E_{2rD}(\mathbb{Q}))$ is equal to one, beside standard methods, there are two other methods of searching for a point of infinite order. One method (see \cite{Silverman3}) is the canonical height search algorithm, the other method (see\cite{Elkies} and \cite{Gross}) is the Heegner point algorithm, both methods are still quite efficiently for conductor of moderately large(say $10^5$ to $10^7$). But both seem quite impractical
in general if $D$ is greater than $10^{8}$.
\end{remark}


\begin{thebibliography}{10}

\bibitem{Bach}E. Bach, and J. Sorenson, Explicit bounds for primes in residue classes, Math. Comp., 65(216): 1717 - 1735, 1996.
\bibitem{Birch}B.J. Birch,  and N.M. Stephens, The parity of the rank of the Mordell-Weil group, Topology 5, 295-299 (1966).

\bibitem{Bremner}A. Bremner, J.H. Silverman and N. Tzanakis, Integral Points in Arithmetic Progression on $y^2=x(x^2-n^2)$, Journal of Number Theory 80, 187-208(2000).

\bibitem{Burhanuddin}I.A. Burhanuddin and Ming-Deh A.Huang, Factoring integers and computing elliptic curve rational points.
Computer Science Technical Report 06-875, Preprint.

\bibitem{Dokchitser} T. Dokchitser, Notes on the parity conjecture, http://arxiv.org/abs/1009.5389.

\bibitem{Edixhoven}B. Edixhoven, Rational elliptic curves are modular ( after Breuil, Conrad, Diamond and Taylor),
S$\acute{e}$minaire Bourbaki, 871 (2000). In: As$\acute{e}$risque 276 (2002), 161-188.

\bibitem{Elkies}N. Elkies, Heegner point computations, Algorithmic Number Theory (L.M. Adelman, M.-D.
Huang, eds.), ANTS-I, Lecture Notes in Computer Science, vol. 877, 1994, pp. 122-133. MR
96f:11080

\bibitem{Gross}B. Gross and D. Zagier, Heegner points and derivatives of L-series, Invent. Math. 84 (1986),
225-320. MR 87j:11057

\bibitem{Lang}S. Lang, S. Conjectured Diophantine estimates on elliptic curves. Arithmetic and geometry, Vol. I, Progr. Math., vol. 35 (1983), 155-171.

\bibitem{Siksek}S. Siksek, Infinite descent on elliptic curves, Rock Mountain Journal of Mathematics, Vol. 25, No. 4 (1995), 1501-1538.

\bibitem{Silverman1}J.H. Silverman, The arithmetic of elliptic curves, Graduate Texts in Mathematics 106,
Springer-Verlag, New York, 1986.

\bibitem{Silverman2}J.H. Silverman, Advances Topics in the Arithmetic of Elliptic Curves, Springer, New York, 1994.

\bibitem{Silverman3}J. H. Silverman, Computing rational points on rank 1 elliptic curves via L-series and canonical heights. Math. Comp. 68 (1999), no. 226, 835¨C858.


\end{thebibliography}
\end{document}